\documentclass{amsart}
\usepackage{amsfonts}
\usepackage{amsmath,amsthm}
\usepackage{amsfonts,amssymb}
\usepackage{mathrsfs}
\usepackage{enumerate}

\newtheorem{thm}{Theorem}[section]

\newtheorem{lem}[thm]{Lemma}

\theoremstyle{remark}
\newtheorem{rem}[thm]{Remark}

\numberwithin{equation}{section}

\def\vz{\varepsilon}
\def\oz{\omega}
\def\lz{\lambda}

\def\az{\alpha}

\def\tz{\theta}

\def\({\Bigl(}
\def \){ \Bigr)}

\def\x{{\bf x}}
\def\h{{\bf h}}
\def\RR{\mathcal{R}}

\def\va{\varepsilon}

\def\y{{\bf y}}
\def\j{{\bf j}}

\begin{document}

\title[] {Average Case $(s, t)$-weak tractability of non-homogenous tensor product problems}

\author{Jia Chen} \address{School of Mathematical Sciences, Capital Normal
University, Beijing 100048,
 China.}
\email{jiachencd@163.com}
\author{Heping Wang} \address{School of Mathematical Sciences, Capital Normal
University, Beijing 100048,
 China.}
\email{wanghp@cnu.edu.cn.}
\author{Jie Zhang}\address{School of Mathematical Sciences, Beijing Normal
University, Beijing 100875,
 China.}
\email{zhangjie91528@163.com.}



\thanks{
 Supported by the
National Natural Science Foundation of China (Project no.
11671271) and
 the  Beijing Natural Science Foundation (1172004)
 }

\begin{abstract} We study $d$-variate  problem in the average case
setting with respect to a zero-mean Gaussian measure. The
covariance kernel of this Gaussian measure is a product of
univariate kernels and satisfies  some special properties. We
study $(s, t)$-weak tractability  of this multivariate problem,
and obtain  a necessary and sufficient condition for  $s>0$ and
$t\in(0,1)$. Our result  can apply to  the  problems with
covariance kernels corresponding to Euler and Wiener integrated
processes, Korobov kernels,  and analytic Korobov kernels.

\end{abstract}

\maketitle
\input amssym.def

\section{Introduction}
Recently, there has been an increasing interest in $d$-variate
 problems  with large or even huge $d$.  Examples
include problems in computational finance, statistics and physics.
In this paper we investigate multivariate  problems
$S=\{S_d\}_{d\in\Bbb N}$ in the average case setting, where
$S_d\,:\,F_d\to G_d $, $F_d$ is a separable Banach space equipped
with a zero-mean Gaussian measure $\mu_d$ , $G_d$ is a Hilbert
space. We only consider continuous linear functional. We use
either {\it the absolute error criterion (ABS)} or {\it the
normalized error criterion (NOR)}. The information complexity
$n^{X}(\vz,S_d)$  is defined as the minimal number of continuous
linear functionals needed to find an $\va$-approximation of $S_d$
for $X\in \{{\rm ABS,\, NOR}\}$.

An algorithm $A\,:\,F_d\to G_d$ is said to be an
$\va$-approximation of $S_d$ for $X\in \{{\rm ABS,\, NOR}\}$  if
$$\bigg(\int_{F_d}\|S_d(f)-A(f)\|_{G_d}^2\mu_d(df)\bigg)^{\frac{1}{2}}\le\va
CRI_d,$$where
\begin{equation*}
CRI_d=\left\{\begin{matrix}
 & \qquad1, \ \quad\qquad \qquad\qquad\qquad\text{ for X=ABS,} \\
 & \big(\int_{F_d}\|S_d(f)\|_{G_d}^2\mu _d(df)\big)^{1/2},\quad \text{ for X=NOR.}
\end{matrix}\right.
\end{equation*}

Tractability of multivariate problems $S$ is concerned with the
behavior of the information complexity $n^X(\vz,S_d)$ for $X\in
\{{\rm ABS,\, NOR}\}$ when the accuracy $\vz$ of approximation
goes to zero and the number $d$ of variables goes to infinity.
Various notions of tractability have been studied recently for
many multivariate problems.  We briefly recall some of the basic
tractability notions (see \cite{NW1, NW2, NW3, S1, SiW}).

Let $S=\{S_d\}_{d\in\Bbb N}$. For $X\in \{{\rm ABS,\, NOR}\}$, we
say $S$ is

$\bullet$   {\it strongly polynomially tractable (SPT)}   iff
there exist non-negative numbers $C$ and $p$ such that for all
$d\in \Bbb N,\ \va \in (0,1)$,
\begin{equation*}
n^X(\va ,S_d)\leq C(\va ^{-1})^p;
\end{equation*} The exponent of SPT  is defined to be  the infimum of all $p$ for which
the above inequality holds;

 $\bullet$  {\it polynomially tractable
(PT)} iff there exist non-negative numbers $C, p$ and $q$ such
that for all $d\in \Bbb N, \ \va \in(0,1)$,
\begin{equation*}
n^X(\va ,S_d)\leq Cd^q(\va ^{-1})^p;
\end{equation*}

$\bullet$   {\it quasi-polynomially tractable (QPT)} iff there
exist two constants $C,t>0$ such that for all $d\in \Bbb N, \ \va
\in(0,1)$,
\begin{equation*}
n^X(\va ,S_d)\leq C\exp(t(1+\ln\va ^{-1})(1+\ln d));
\end{equation*}

$\bullet$ {\it uniformly weakly tractable (UWT)} iff for all $s,
t>0$,
\begin{equation*}
\lim_{\varepsilon ^{-1}+d\rightarrow \infty }\frac{\ln n^X(\va
,S_d)}{(\va ^{-1})^{s }+d^{t }}=0;
\end{equation*}

$\bullet$  {\it weakly tractable (WT)} iff
\begin{equation*}
\lim_{\va ^{-1}+d\rightarrow \infty }\frac{\ln n^X(\va ,S_d)}{\va
^{-1}+d}=0;
\end{equation*}

$\bullet$ {\it $(s,t)$-weakly tractable ($(s,t)$-WT)} for positive
$s$ and $t$ iff
\begin{equation*}
\lim_{\varepsilon ^{-1}+d\rightarrow \infty }\frac{\ln n^X(\va
,S_d)}{(\va ^{-1})^{s }+d^{t }}=0.
\end{equation*}

This paper is devoted to studying average case $(s,t)$-weak
tractability of  non-homogenous tensor product problems with
covariance kernels corresponding to Euler and Wiener integrated
processes, Korobov kernels,  and analytic Korobov kernels. Such
problems were investigated in \cite{S3}  for Euler and Wiener
integrated processes under NOR, and in \cite{LX2} for analytic
Korobov case under NOR and ABS.
  The authors in \cite{S3, LX2} obtained that $(s,t)$-WT always
  holds with $s>0$ and $t>1$, and $(s,1)$-WT with $s>0$ holds iff
  WT holds.
 However,   they  did not obtain the
matching necessary and sufficient conditions on $(s,t)$-WT with
$s>0$ and $t\in(0,1)$. The matching necessary and sufficient
condition on $(s,t)$-WT with $s>0$ and $t\in(0,1)$ was first
obtained in \cite{CW} for average case multivariate approximation
with Gaussian covariance kernels.

In this paper, we use a unified method to get  a necessary and
sufficient condition for $(s, t)$-WT for $t\in(0,1)$ and $s>0$.
Specially for Euler and Wiener integrated processes, the measures
$\mu_d$  are defined in terms of the nondecreasing sequence
$\{r_k\}_{k\in \Bbb N}$ of nonnegative integers
$$0\le r_1\le r_2\le r_3 \le \dots.$$
Roughly speaking, $r_k$ measures the smoothness of the process
with respect to the $k$th variable. For the normalized error
criterion,  we obtain for $t\in(0,1)$ and $s>0$,

$\bullet$ for the Euler integrated process,
$$(s,t)-WT \
\Leftrightarrow\
\lim\limits_{k\to\infty}k^{1-t}3^{-2r_k}(1+r_k)=0;$$

$\bullet$ for the Wiener integrated process,
$$(s,t)-WT \ \Leftrightarrow\ \lim\limits_{k\to\infty}k^{1-t} (1+r_k)^{-2}\ln ^+
(1+r_k)=0,$$where $\ln^+x=\max(\ln x,1).$

The paper is organized as follows. In Section 2 we give the
preliminaries about
 non-homogeneous tensor product problems  in the
average case setting and  present the main results, i.e.,  Theorem
2.1.
   Section 3 is
devoted to proving Theorem 2.1.
  In  Section 4, we give the applications of Theorem 2.1 to  the  problems with covariance
kernels corresponding to Euler and Wiener integrated processes,
Korobov kernels,  and analytic Korobov kernels.

\section{Preliminaries and main results}

We recall the concept of non-homogeneous linear multivariate
tensor product problems in average case setting, see \cite{LPW1}.

 Let $F_d, H_d$  are given by tensor products. That
is,
\begin{equation*}
F_d=F^{(1)}_1\otimes F^{(1)}_2\otimes \dots \otimes F^{(1)}_d
\quad \text {and} \quad H_d=H^{(1)}_1\otimes  H^{(1)}_2\otimes
\dots \otimes H^{(1)}_d,
\end{equation*}
where Banach spaces $F^{(1)}_k$ are of univariate real functions
equipped with a zero-mean Gaussian measure $\mu^{(1)}_k$, and
$H^{(1)}_k$ are Hilbert spaces, $k=1,2,\dots,d$. We set
$$S_d=S^{(1)}_1\otimes S^{(1)}_2\otimes \dots \otimes S^{(1)}_d,\ \ \mu_d=\mu_1^{(1)}\otimes \mu_2^{(1)}\otimes \dots \otimes \mu_d^{(1)},$$
where
$$S^{(1)}_k=F^{(1)}_k\to H^{(1)}_k,\quad k=1,2,\dots,d$$ are continuous linear
operators. Then $\mu_d$ is a zero-mean Gaussian measure on $F_d$
with covariance operator $C_{\mu_d}: F_d^*\to F_d$.

Let $\nu_d =\mu_d (S_d)^{-1}$ be the induced measure. Then $\nu_d$
is a zero-mean Gaussian measure on $H_d$ with covariance operator
$C_{\nu_d}: H_d\to H_d$ given by $$ C_{\nu_d}=S_d
\,C_{\mu_d}\,S_d^*,$$where $S_d^*:H_d\to F_d^*$ is the operator
dual to $S_d$.
 Let
$\nu^{(1)}_k=\mu^{(1)}_k(S^{(1)}_k)^{-1}$ be the induced zero-mean
Gaussian measure on $H^{(1)}_k$, and let
$C_{\nu^{(1)}_k}:H^{(1)}_k\to H^{(1)}_k$ be the covariance
operator of the  measure $\nu^{(1)}_k$. Then
$$ \nu_d=\nu_1^{(1)}\otimes \nu_2^{(1)}\otimes \dots \otimes \nu_d^{(1)},\ \ {\rm and}\ \ C_{\nu _d}=C_{\nu^{(1)}_1}\otimes C_{\nu^{(1)}_2}\otimes \dots \otimes C_{\nu^{(1)}_d}.$$
 The eigenpairs of
$C_{\nu^{(1)}_k}$ are denoted by $\big\{(\lz(k,j),
\eta(k,j))\big\}_{j\in \Bbb N}$, and satisfy
$$C_{\nu^{(1)}_k}(\eta(k,j))=\lz(k,j)\eta(k,j), \ {\rm with}\ \lz(k,1)\geq \lz(k,2)\geq \dots \geq 0.$$
Then  $${\rm
trace}(C_{\nu^{(1)}_k})=\int_{H^{(1)}_k}\|f\|^2_{H^{(1)}_k}\nu^{(1)}_k(df)
=\sum_{j=1}^\infty \lz(k,j) <\infty.$$ The eigenpairs of $C_{\nu
_d}$ are given by
$$\big\{(\lz _{d,\j},\eta _{d,\j})\big\}_{\j=(j_1,j_2,\dots,j_d)\in \Bbb N^d},$$
where
$$\lz_{d,\j}=\prod _{k=1}^d\lz (k,j_k)\quad \text{and}\quad \eta_{d,\j}=\prod_{k=1}^d\eta(k,j_k).$$

Let the sequence $\{\lz _{d,j}\}_{j\in \Bbb N}$ be the
non-increasing rearrangement of $\{\lz_{d,\j}\}_{\j\in \Bbb N^d}$.
Then we have
\begin{equation*}
\sum_{j\in \Bbb N}\lz^{\tau}_{d,j}=\prod_{k=1}^d\sum_{j=1}^\infty
\lz(k,j)^\tau,\quad \text{for any}\quad \tau>0.
\end{equation*}

We approximate $S_d \, f$ by algorithms $A_{n,d}$  that use only
finitely many continuous linear functionals. A function $f\in F_d$
is approximated by an algorithm
\begin{equation}\label{2.1}A_{n,d}(f)=\Phi
_{n,d}(L_1(f),L_2(f),\dots,L_n(f)),\end{equation} where
$L_1,L_2,\dots,L_n$ are continuous linear functionals on $F_d$,
 and $\Phi _{n,d}:\;\Bbb R^n\to
H_d$ is an arbitrary measurable mapping. The average case error
for $A_{n,d}$ is defined by
\begin{equation*}
e(A_{n,d})\;=\;\Big( \int _{F_d}\big \| S_d\,f-A_{n,d}f \big
\|_{H_d}^{2}\mu _d(df) \Big )^{\frac{1}{2}}.
\end{equation*}

The $n$th minimal average case error, for $n\ge 1$, is defined and
given by (see \cite{NW1})
\begin{equation*}
e(n,d)=\inf_{A_{n,d}}e(A_{n,d} )= \Big (\sum
_{j=n+1}^{\infty}\lambda _{d,j}  \Big )^{\frac{1}{2}},
\end{equation*}
where the infimum is taken over all algorithms of the form
\eqref{2.1}. It is achieved by the $n$th optimal algorithm
\begin{equation*}
A_{n,d}^{*}(f)=\sum_{j=1}^n  \big\langle f,\eta _{d,j} \big
\rangle _{H_d}\eta _{d,j}.\end{equation*}

For $n=0$, we use $A_{0,d}=0$. We remark that  the so-called
initial error $e(0,d)$ is defined and given by
\begin{equation*}
e(0,d)=\Big ( \int _{F_d}\big \| S_d\,f \big \|_{H_d}^{2}\mu
_d(df) \Big )^{\frac{1}{2}}=\Big (\sum _{j=1}^{\infty}\lambda
_{d,j}  \Big )^{\frac{1}{2}}.
\end{equation*}

 The information
complexity for $S_d$ can be studied using either the absolute
error criterion (ABS), or the normalized error criterion (NOR).
Then we define the information complexity $n^X(\va ,S_d)$ for
$X\in \{ {\rm ABS,\, NOR}\}$ as
\begin{equation*}
n^{ X}(\va ,S_d)=\min\{n:\,e(n,S_d)\leq \va CRI_d\},
\end{equation*}where
\begin{equation*}
CRI_d=\left\{\begin{matrix}
 & 1, \; \quad\qquad\text{ for X=ABS,} \\
 &e(0,S_d),\quad \text{ for X=NOR.}
\end{matrix}\right.
\end{equation*}

\vskip 3mm

In this paper we consider a special class of non-homogeneous
tensor product problems $S=\big\{S_d\big\}_{d\in\Bbb N}$. Assume
that the eigenvalues
$$\Big\{\prod_{k=1}^d\lz(k,j_k)\Big\}_{(j_1,j_2,\dots,j_d)\in \Bbb
N^d}$$ of the covariance operator $C_{\nu_d}$ of the problem $S$
satisfy the following three conditions:

(1) $\ \ \ \lz(k,1)=1,\  \ k\in \Bbb N;$

\vskip 1mm

(2) there exist a decreasing positive sequence
$\big\{f_k\big\}_{k\in\Bbb N}$ and two positive constants $A_2\in
(0,1],\  A_1\ge 1$ such that for all $k\in\Bbb N$, we have
$$A_2f_k\le h_k\le A_1f_k,$$ where $h_k=\frac{\lz(k,2)}{\lz(k,1)}\in
(0,1];$

\vskip 1mm (3) there exist two constants $\tau_0\in(0,1)$ and
$M_{\tau_0}$ for which
$$\sup_{k\in\Bbb N} H(k,\tau_0)\le M_{\tau_0}<\infty,$$
where
$$H(k,x):=\sum_{j=2}^\infty\big(\frac{\lz(k,j)}{\lz(k,2)}\big)^x.$$

\vskip 1mm \noindent Then we say that  the  problem
$S=\big\{S_d\big\}_{d\in\Bbb N}$ has Property (P).

\

We make some comments on Property (P). Usually, the sequence
$\{h_k\}$ in  Condition (2)  is decreasing. In this case,
$A_1=A_2=1$, and $f_k=h_k,\ k\in \Bbb N$. For the problem $S$ with
Property (P), we have for $\vz\in(0,1)$ and $d\in\Bbb N$,
\begin{equation}\label{2.1-0}n^{\rm ABS}(\vz,d)\ge n^{\rm NOR}(\vz,d).\end{equation}

 Note that for any $x\ge
\tau_0,\ k\in\Bbb N$, $$  \ln
\Big(\sum_{k=1}^d\lz_{d,k}^x\Big)=\sum_{k=1}^d\ln(1+h_k^xH(k,x)),\
\ {\rm and}\ \ 1\le H(k,x)\le H(k,\tau_0)\le M_{\tau_0}<\infty.$$
According to  Conditions (2) and (3), we have for any $x\ge
\tau_0$,
\begin{equation}\label{2.1-1} \ln2\sum_{k=1}^dh_k^x\le \sum_{k=1}^d\ln(1+h_k^x)\le \ln
\Big(\sum_{k=1}^d\lz_{d,k}^x\Big)\le \sum_{k=1}^d\ln(1+M_{\tau_0}
h_k^x)\le M_{\tau_0}\sum_{k=1}^dh_k^x,\end{equation} where in the
first inequality we used the inequality $\ln(1+x)\ge x\ln2,\
x\in[0,1]$, and in the last inequality we used the inequality
$\ln(1+x)\le x, \ x>0$.

We are ready to present the main result of this paper.

\begin{thm}\label{thm1}
Let $S=\big\{S_d\big\}_{d\in\Bbb N}$ be a non-homogeneous tensor
product problem with Property (P). Then for NOR or ABS,
$(s,t)$-WT holds  with $s>0$ and $t\in(0,1)$ iff
\begin{equation} \label{2.1-2}
\lim\limits_{k\to\infty}k^{1-t}f_k\ln^+\frac1{f_k}=0.
\end{equation}
\end{thm}

\vskip 3mm

\begin{rem} Let $S=\big\{S_d\big\}_{d\in\Bbb N}$ be a non-homogeneous tensor
product problem with Property (P). Using the method of \cite{LX2,
S2}, we can obtain that  for ABS or NOR,   $(s,t)$-WT always holds
with $s>0$ and $t>1$, and $(s,1)$-WT holds with $s>0$ iff WT holds
iff $$\lim\limits_{k\to\infty}f_k=0.$$
\end{rem}

\begin{rem} Let $S=\big\{S_d\big\}_{d\in\Bbb N}$ be a non-homogeneous tensor
product problem. If the eigenvalues of the covariance operator
$C_{\nu_d}$ of the problem $S$ satisfy Conditions (2) and (3),
then for NOR, $(s,t)$-WT holds with $s>0$ and $t\in(0,1)$ iff
\eqref{2.1-2} holds.

Indeed, let $\tilde S=\big\{\tilde S_d\big\}_{d\in\Bbb N}$ be the
non-homogeneous tensor product problem which  the eigenvalues
$\big\{\prod_{k=1}^d\tilde
\lz(k,j_k)\big\}_{(j_1,j_2,\dots,j_d)\in \Bbb N^d}$ of the
corresponding covariance operator $C_{\tilde \nu_d}$ of the
induced measure of  $\tilde S$  satisfy $$ \tilde \lz(k,j)=\frac
{\lz(k,j)}{ \lz(k,1)},\ \  j\in\Bbb N,\ k=1,\dots,d.$$Then $\tilde
S$ has Property (P) with the same $h_k$. Also for NOR, the
problems $S$ and $\tilde S$ have the same tractability. Hence, for
NOR, $(s,t)$-WT holds with $s>0$ and $t\in(0,1)$ iff \eqref{2.1-2}
holds.

\end{rem}

In order to prove Theorem \ref{thm1}, we need the following lemma.

\begin{lem}\label{lem1}
Let $S=\big\{S_d\big\}_{d\in \Bbb N}$ be a non-homogeneous tensor
product problem. Then for NOR, we have for $x>0$
\begin{equation*}
n^{\rm NOR}(\va,S_d)\ge
(1-\va^2)^{\frac{x+1}{x}}\Big(\prod_{k=1}^d\frac{1+h_k}{1+h_k^{x+1}}\Big)^{\frac{1}{x}},
\end{equation*}
where $$h_k=\frac{\lz(k,2)}{\lz(k,1)}\in(0,1].$$
\end{lem}

\begin{proof}
We set $$n=n^{\rm NOR}(\va,S_d),\ \ \ \
\overline{\lz}_{d,k}=\frac{\lz_{d,k}}{\sum_{k=1}^\infty\lz_{d,k}}.$$
It follows from the definition of $n^{\rm NOR}(\va,S_d)$ that
$$1-\sum_{k=1}^n\overline{\lz}_{d,k}=\sum_{k=n+1}^\infty\overline{\lz}_{d,k}\le\va^2.$$
We have
\begin{equation}\label{2.2-0}1-\va^2\le\sum_{k=1}^n\overline{\lz}_{d,k}\le
n^{\frac{x}{x+1}}\,\Big(\sum_{k=1}^n\overline{\lz}_{d,k}^{\
x+1}\Big)^\frac{1}{x+1}\le n^{\frac{x}{x+1}}\, \Big(\sum
_{k=1}^\infty\overline{\lz}_{d,k}^{\
x+1}\Big)^\frac{1}{x+1},\end{equation} which leads to
\begin{equation}\label{2.2}
\begin{split}
n^{\rm NOR}(\va,S_d)&=n\ge(1-\va^2)^\frac
{x+1}{x}\Big(\sum_{k=1}^n\overline{\lz}_{d,k}^{\
x+1}\Big)^\frac{-1}{x}\\
&=(1-\va^2)^\frac
{x+1}{x}\bigg(\prod_{k=1}^d\frac{(\sum_{j=1}^\infty\lz(k,j))^{x+1}}{\sum_{j=1}^\infty(\lz(k,j))^{x+1}}\bigg)^{\frac
1{x}}.
\end{split}
\end{equation}
We note that for $k=1,2,\dots,d, \ j\ge3$ and $x>0$,
$$\lz(k,j)(\lz(k,i))^{x+1}\ge(\lz(k,j))^{x+1}\lz(k,i),\ \ i=1, 2.$$
It follows that
$$\lz(k,j)\big((\lz(k,1))^{x+1}+(\lz(k,2))^{x+1}\big)\ge(\lz(k,j))^{x+1}\big(\lz(k,1)+\lz(k,2)\big),$$
and so
$$\sum_{j=3}^\infty\lz(k,j)\big((\lz(k,1))^{x+1}+(\lz(k,2))^{x+1})\big)\ge\sum_{j=3}^\infty(\lz(k,j))^{x+1}\big(\lz(k,1)+\lz(k,2)\big).$$
This implies that
$$1+\frac{\sum_{j=3}^\infty\lz(k,j)}{\lz(k,1)+\lz(k,2)}\ge
1+\frac{\sum_{j=3}^\infty(\lz(k,j))^{x+1}}{(\lz(k,1))^{x+1}+(\lz(k,2))^{x+1}}.$$
Hence, we have
$$\Big(\frac{\sum_{j=1}^\infty\lz(k,j)}{\lz(k,1)+\lz(k,2)}\Big)^{x+1}\ge
\frac{\sum_{j=1}^\infty\lz(k,j)}{\lz(k,1)+\lz(k,2)}\ge
\frac{\sum_{j=1}^\infty(\lz(k,j))^{x+1}}{(\lz(k,1))^{x+1}+(\lz(k,2))^{x+1}}.$$
It follows that
$$\frac{\big(\sum_{j=1}^\infty\lz(k,j)\big)^{x+1}}{\sum_{j=1}^\infty(\lz(k,j))^{x+1}}\ge\frac{\big(\lz(k,1)+\lz(k,2)\big)^{x+1}}{(\lz(k,1))^{x+1}+(\lz
(k,2))^{x+1}}=\frac{(1+h_k)^{x+1}}{1+h_k^{x+1}}.$$ By \eqref{2.2}
and the above inequality, we get \begin{align*}n^{\rm
NOR}(\va,S_d)&\ge
(1-\va^2)^{\frac{x+1}{x}}\Big(\prod_{k=1}^d\frac{(1+h_k)^{x+1}}{1+h_k^{x+1}}\Big)^{\frac{1}{x}}\\
&\ge (1-\va^2)^{\frac{x+1}{x}}
\Big(\prod_{k=1}^d\frac{1+h_k}{1+h_k^{x+1}}\Big)^{\frac{1}{x}}.\end{align*}
 Lemma \ref{lem1} is proved.
\end{proof}

\section{Proof of Theorem 2.1}

\noindent{\it \textbf{Proof of Theorem \ref{thm1}}.} \vskip 2mm

We first show that \eqref{2.1-2} holds whenever $(s,t)$-WT holds
with $s>0$ and $t\in (0,1)$ for NOR or ABS. Due to \eqref{2.1-0},
it suffices to prove \eqref{2.1-2} under NOR.

 Assume that $s>0$ and  $t\in(0,1)$. Suppose that $(s,t)$-WT holds  for
 NOR.  We set
\begin{equation}\label{3.3} u_k:=\max(f_k, \frac{1}{2k}),
\qquad \text{and}\qquad
s_k:=\frac{1}{2}\big(\ln^+\frac{1}{u_k}\big)^{-1}, \qquad k\in
\Bbb N,
\end{equation}where $f_k$ is given in Condition (2) of Property
(P).  Then $\{u_k\}$ is monotonically decreasing. We want to show
that $\lim\limits_{j\to \infty}u_j=0$.

 It follows from \eqref{2.2-0} and $\lz_{d,1}=1$ that
\begin{equation*}
1-\va^2\le\sum_{k=1}^{n^{\rm NOR}(\va,S_d)}\overline\lz_{d,k}\le
n^{\rm NOR}(\va, S_d)\,\overline\lz_{d,1}= n^{\rm NOR}(\va,
S_d)\,\Big(\sum_{k=1}^\infty \lz_{d,k}\Big)^{-1}.
\end{equation*}
This implies that
\begin{align}\label{3.4}
\ln n^{\rm NOR}(\va, S_d)\ge\ln(1-\va^2)+\ln
\Big(\sum_{k=1}^\infty \lz_{d,k}\Big).
\end{align}
 By \eqref{2.1-1} and Condition (2) of Property (P), we get
\begin{equation}\label{3.5}\ln
\Big(\sum_{k=1}^\infty \lz_{d,k}\Big)\ge \ln2\sum_{k=1}^dh_k\ge
A_2\ln2\sum_{k=1}^df_k \ge A_2(\ln2) d\,f_d.\end{equation}
 Since
$(s,t)$-WT holds for NOR, we obtain by \eqref{3.4} and \eqref{3.5}
that
$$0=\lim_{d\to\infty}\frac{\ln(n^{\rm NOR}(\frac{1}{2}, S_d))}{\big(\frac{1}{2}\big)^s+d^t}\ge \lim\limits_{d\to\infty}\frac{\ln\frac{3}{4}+A_2(\ln2) df_d}{d^t}
=A_2\ln2\lim_{d\to\infty}d^{1-t}f_d\ge 0,$$ which implies
$\lim\limits_{d\to \infty}d^{1-t}f_d=0$ and hence
$\lim\limits_{d\to \infty}u_d=0$.

 Applying  Lemma \ref{lem1} with $x=s_d>0$,  we
 obtain
\begin{align}
\ln \big(n^{\rm NOR}(1/2,S_d)\big)&\ge
\frac{s_d+1}{s_d}\ln \frac3{4}+{\frac{1}{s_d}}\sum_{k=1}^d\ln\Big(\frac{1+h_k}{1+h_k^{s_d+1}}\Big)\notag\\
&\ge\frac{1}{s_d}\ln\frac{3}{4}+\frac{1}{s_d}\sum_{k=1}^d\big(\frac{h_k-h_k^{s_d+1}}{1+h_k^{s_d+1}}\big)\ln2\notag\\
&\ge\frac{1}{s_d}\ln\frac{3}{4}+\frac{\ln2}{2s_d}\sum_{k=1}^d(h_k-h_k^{s_d+1}),\label{3.6}
\end{align}
where in the second inequality we used the inequality $\ln(1+x)\ge
x\ln 2,\ x\in[0,1]$.

We remark that the function $u(x)=x-x^{1+s_d}$ is monotonically
increasing in $(0,e^{-1})$. Since $\lim\limits_{d\to
\infty}u_d=0$, there exists a positive integer $K$ such that
$0<u_k< e^{-1}$ holds for any $k\ge K$. It follows that
\begin{align} \sum_{k=1}^d(h_k-h_k^{s_d+1})&\ge \sum_{k=K}^d(h_k-h_k^{s_d+1})\notag\\ &\ge \sum_{k=K}^d\big(A_2f_k-(A_2f_k)^{s_d+1}\big) \notag\\ &\ge
(d-K)\big(A_2f_d-(A_2f_d)^{s_d+1}\big).\label{3.7}
\end{align}

 By \eqref{3.3} we get
\begin{equation}\label{3.8}
\frac{1}{s_d}=2\ln^+\big(\frac{1}{u_d}\big)\le 2\ln^+(2d),\ \ {\rm
and}\ \
\lim_{d\to\infty}\frac{1}{s_dd^t}=\lim_{d\to\infty}\frac{2\ln^+(2d)}{d^t}=0.
\end{equation}

 Since
$(s,t)$-WT holds for NOR, we obtain by \eqref{3.6}, \eqref{3.7},
and \eqref{3.8} that
\begin{align}
0&=\lim_{d\to\infty} \frac{\ln \big(n^{\rm
NOR}(1/2,S_d)\big)}{2^s+d^t}\notag\\ &\ge
\lim_{d\to\infty}\Big(\frac{\ln\frac{3}{4}}{s_dd^t}+\frac{(d-K)\ln2}{2s_dd^t}\big(A_2f_d-(A_2f_d)^{s_d+1}\big)\Big)\notag\\&
=\frac{\ln2}2\lim_{d\to\infty}\frac{d^{1-t}}{s_d}\big(A_2f_d-(A_2f_d)^{s_d+1}\big)\ge
0 ,\notag
\end{align}
which yields that
\begin{equation}\label{3.9}\lim_{d\to\infty}\frac{d^{1-t}}{s_d}\Big(A_2f_d-(A_2f_d)^{s_d+1}\Big)=
0.\end{equation}

 Applying  the mean value theorem to the
function $\phi(x)=a^x,\ a\in(0,1)$, we obtain for some $\tz\in(0,
1)$,
\begin{equation}\label{3.10}
a^{s_d}a s_d\ln\big(\frac{1}{a}\big)\le a-a^{1+s_d}=a^{1+\tz s_d}s_d\ln\big(\frac{1}{a}\big)\le  a s_d\ln\big(\frac{1}{a}\big)\\
\end{equation}
We get by \eqref{3.10} that
\begin{equation*}0\le
\lim_{d\to\infty}\frac{d^{1-t}}{s_d}\Big(\frac{A_2}{2d}-\big(\frac{A_2}{2d}\big)^{s_d+1}\Big)\le
\lim_{d\to\infty}{d^{1-t}}\big(\frac{A_2}{2d}\big)\,\ln\big(\frac{2d}{A_2}\big)=
0,\end{equation*} which combining with \eqref{3.9}, gives that
\begin{equation}\label{3.11}\lim_{d\to\infty}\frac{d^{1-t}}{s_d}\Big(A_2u_d-(A_2u_d)^{s_d+1}\Big)=
0.\end{equation} Noting that $$\lim_{d\to\infty} \big({A_2u_d}
\big)^{s_d}=\lim_{d\to\infty}\exp\big(-\frac{\ln
\frac{1}{A_2}+\ln\frac1{u_d}}{2\ln^+\frac1{u_d}}\Big)=e^{-1/2},
$$by \eqref{3.10} we have
\begin{align*}0&=\lim_{d\to\infty}\frac{d^{1-t}}{s_d}\big(A_2u_d-(A_2u_d)^{s_d+1}\big)
\\ &\ge \lim_{d\to\infty}d^{1-t} \big({A_2u_d}
\big) \big({A_2u_d}
\big)^{s_d}\ln\big(\frac{1}{A_2u_d} \big)\\
&\ge e^{-1/2}A_2\lim_{d\to\infty}d^{1-t} {u_d}
\ln^+\big(\frac{1}{u_d}\big)\ge0, \end{align*}which implies that
$$\lim_{d\to\infty} d^{1-t}u_d
\ln^+\big(\frac{1}{u_d}\big)=0.$$ Hence, we conclude from the
monotonicity of the function $\varphi(x)=x\ln^+
\big(\frac1x\big)=x\ln\big(\frac1x\big),\ x\in(0,1/e)$ that
$$0\le  \lim_{d\to\infty} d^{1-t}f_d
\ln^+\frac{1}{f_d}\le \lim_{d\to\infty} d^{1-t}u_d
\ln^+\frac{1}{u_d}=0,$$giving \eqref{2.1-2}.

Next we show that $(s,t)$-WT with $s>0$ and $t\in(0,1)$ holds for
NOR or ABS whenever \eqref{2.1-2} holds. Due to \eqref{2.1-0}, it
suffices to prove $(s,t)$-WT holds for ABS.

We have for $\tau \in (0,1)$,
\begin{equation}\label{3.12}
\sum_{k=n+1}^\infty\lz_{d,k}\le\lz_{d,n+1}^\tau\sum_{k=n+1}^\infty\lz_{d,k}^{1-\tau}\le\lz_{d,n+1}^\tau\sum_{k=1}^\infty\lz_{d,k}^{1-\tau}.
\end{equation}
Since
$$(n+1)\lz_{d,n+1}^{1-\tau}\le\sum_{k=1}^{n+1}\lz_{d,k}^{1-\tau}\le\sum_{k=1}^\infty\lz_{d,k}^{1-\tau},$$
we get
$$\lz_{d,n+1}\le(n+1)^{-\frac{1}{1-\tau}}\big(\sum_{k=1}^\infty\lz_{d,k}^{1-\tau}\big)^\frac 1{1-\tau},$$
which combining with \eqref{3.12} yields
\begin{equation}\label{3.13}
\sum_{k=n+1}^\infty\lz_{d,k}\le (n+1)^{-\frac
\tau{1-\tau}}\big(\sum_{k=1}^\infty\lz_{d,k}^{1-\tau}\big)^{\frac
1{1-\tau}}.
\end{equation}
Setting $$ n=\left
\lfloor\va^{\frac{-2(1-\tau)}{\tau}}\Big(\sum_{k=1}^\infty
\lz_{d,k}^{1-\tau}\Big)^{\frac{1}{\tau}}\right \rfloor$$ in
\eqref{3.13}, we have $$\sum_{k=n+1}^\infty\lz_{d, k}\le\va^2.$$
It follows from the definition of $n^{\rm ABS}(\va,S_d)$ that
\begin{equation}\label{3.14}n^{\rm ABS}(\va,S_d)\le\left \lfloor\va^{\frac{-2(1-\tau)}{\tau}}\Big(\sum_{k=1}^\infty \lz_{d,k}^{1-\tau}\Big)^{\frac{1}{\tau}}
\right
\rfloor\le\va^{\frac{-2(1-\tau)}{\tau}}\Big(\sum_{k=1}^\infty
\lz_{d,k}^{1-\tau}\Big)^{\frac{1}{\tau}}.\end{equation}We let
$\tau=s_d$, where   $s_k.\, u_k$ are given in \eqref{3.3}. By
\eqref{3.14} and \eqref{2.1-1}  we obtain
\begin{equation*}
\begin{split}
\ln n^{\rm ABS}(\va,S_d)&\le \frac{2(1-s_d)}{s_d}\ln \va^{-1}+\frac 1{s_d}\ln \big(\sum_{k=1}^\infty\lz_{d,k}^{1-s_d}\big)\\
&\le \frac{2}{s_d}\ln \va^{-1}+\frac {M_{\tau_0}}{s_d} \sum_{k=1}^dh_k^{1-s_d}\\
&\le \frac{2}{s_d}\ln \va^{-1}+\frac {M_{\tau_0}A_1^{1-s_d}}{s_d}
\sum_{k=1}^du_k^{1-s_d}.
\end{split}
\end{equation*}
 Noting
that $A_1>1$ and $$u_k^{-s_d}=\exp\Big(\frac{\ln
\frac1{u_k}}{2\ln^+\frac1{u_k}}\Big)\le e^{1/2}, \
k=1,2,\dots,d,$$ we continue to get
\begin{equation}\label{3.15} \ln n^{\rm ABS}(\va,S_d)\le \frac{2}{s_d}\ln \va^{-1}+\frac {e^{1/2}M_{\tau_0}A_1}{s_d}
\sum_{k=1}^du_k.\end{equation}

Assume that \eqref{2.1-2} holds.  Note that
$$\lim\limits_{d\to\infty}d^{1-t}(\frac1{2d})\ln^+(2d)=0.$$ It follows from the monotonicity of the function $\varphi
(t)=t\ln^+(\frac1{t}),\ t\in(0,1/e)$ that
$$\lim\limits_{d\to\infty}d^{1-t}u_d\ln^+\frac1{u_d}=0.$$
It follows from \eqref{3.8} that
$$0\le \lim\limits_{\va^{-1}+d\to\infty}\frac{\frac{2}{s_d}\ln
\va^{-1}}{\va^{-s}+d^t}\le\lim\limits_{\va^{-1}+d\to\infty}\frac{s_d^{-2}+(\ln
\va^{-1})^2}{\va^{-s}+d^t}=0.$$  In order to show that $(s,t)$-WT
holds for ABS, by \eqref{3.15} we only need to prove
\begin{equation}\label{3.16}\lim\limits_{d\to\infty}\frac1{d^ts_d}\,{\sum_{k=1}^du_k}=0.\end{equation}

We know that $\varphi (t)=t\ln(\frac1{t})$ is  monotonically
increasing  in $(0,e^{-e})$. So the inverse function $\varphi
^{-1}(t)$ is also monotonically increasing in $t\in(0,e^{1-e})$.
Let $y=\varphi (t)=t\ln(\frac1{t}),\ t\in(0,e^{-e})$.  Then we
have
$$\ln(\frac{1}{y})=
\frac12\ln\frac1{t}
+(\frac12\ln\frac1{t}-\ln(\ln\frac1{t})\big)\ge
\frac12\ln\frac1{t},$$ since $\psi(x)=\frac x2 -\ln x$ is
increasing in $[2,\infty)$ and hence $$\psi(\ln\frac1{t})\ge
\psi(e)=e/2-1>0,\ t\in(0,e^{-e}).$$We get further
\begin{equation}\label{3.17}
 t=\varphi ^{-1}(y)=\frac y{\ln\frac1{t}}\le \frac{2y}{\ln\frac1{y}}.\end{equation}
Since $$\lim\limits_{d\to\infty}d^{1-t}u_d\ln^+\frac1{u_d}=0,$$
for any $\va\in(0,1)$ there exists an  integer $K_1\ge 4$ such
that for all $k\ge K_1$, $$0<u_k< e^{-e},\ \ {\rm and}\ \
k^{1-t}u_k\ln\frac1{u_k}\le \va.$$ This yields
$$\varphi (u_k)\le \va k^{t-1}. $$It follows from \eqref{3.17} that
\begin{align*}u_k \le \varphi ^{-1}({\va k^{t-1}})\le \frac{2\va k^{t-1}}{\ln\frac1{\va k^{t-1}}}
=\frac{2\va k^{t-1}}{\ln\va^{-1}+(1-t)\ln k}\le \frac{2\va
k^{t-1}}{(1-t)\ln k}.\end{align*} We notice that
$v(x)=\frac{x^{t/2}}{\ln x}$ is increasing in $[e^{2/t},\infty)$
due to the fact that $$v'(x)=\frac{x^{t/2-1}}{\ln ^2x}\big(\frac
t2\ln x-1\big)\ge 0.$$It follows from \eqref{3.8} that for the
above $\vz\in(0,1)$ there exists a positive integer $K_2$ for
which \begin{equation}\label{3.18} \frac{1}{s_dd^t}\le
\vz\end{equation} holds for any $k\ge K_2$. We set
$$K=\max(K_1,\lfloor e^{2/t}\rfloor+1, K_2).$$ Then for any $d>K$, we
have
\begin{align}
\sum_{k=1}^du_k&\le\sum_{k=1}^{K}\frac1{A_2}\max(h_k,\frac1{2k})+\sum_{k=K+1}^d\frac{2\va k^{t/2-1}k^{t/2}}{(1-t)\ln k}\notag\\
&\le\sum_{k=1}^{K}\frac1{A_2}+\frac{2\va d^{t/2}}{(1-t)\ln d}\sum_{k=K+1}^dk^{t/2-1}\notag\\
&\le \frac K{A_2}+\frac{4\va d^t}{t(1-t)\ln d},\label{3.19}
\end{align}
where in the last inequality  we used the inequality
$$\sum_{k=K+1}^dk^{t/2-1}\le \sum_{k=1}^dk^{t/2-1}\le \sum_{k=1}^d\int_{k-1}^kx^{t/2-1}dx\le \int_0^dx^{t/2-1}dx=\frac2t d^{t/2}.  $$
It follows from \eqref{3.19}, \eqref{3.18}, and \eqref{3.8} that
\begin{align*}
\frac1{d^ts_d}\,{\sum_{k=1}^du_k}&\le \frac K{A_2}\frac1{d^ts_d}+\frac{4\va }{t(1-t)s_d\ln d}\\
&\le \frac {K\vz}{A_2} +\frac{8\va \ln(2d)}{t(1-t)\ln d}\\ &\le
\vz \Big(\frac {K}{A_2}+ \frac{16}{t(1-t)}\Big).
\end{align*}
This gives \eqref{3.16}.  We conclude that $(s,t)$-WT holds for
ABS.

The proof of  Theorem \ref{thm1} is completed. $\hfill\Box$

\section{Applications of Theorem 2.1}

 Consider the
approximation problem $S=\{S_d\}_{d\in\Bbb N}$,
$$S_d\,\,: \,\,C([0,1]^d)\to L_2([0,1]^d)\quad \text{with} \quad S_d(f)=f.$$
The space $C([0,1]^d)$ of continuous real functions is equipped
with a zero-mean Gaussian measure $\mu_d$ whose covariance kernel
is given by
\begin{equation*}
K_d(\x,\y)=\int_{C([0,1]^d)}f(\x)f(\y)\mu_d(df),\quad \x,\, \y \in
[0,1]^d.
\end{equation*}The covariance kernels $K_d(\x,\y)$ are of tensor
product and  correspond to Euler and Wiener integrated processes,
Korobov kernels,  and analytic Korobov kernels. This section  is
devoted to giving the applications of Theorem 2.1 to these cases.

\subsection{$(s,t)$-WT of Euler and Wiener
integrated processes}

\

In this subsection we consider  multivariate approximation
problems $S=\{S_d\}$ defined over the space $C([0,1]^d)$ equipped
with  zero-mean Gaussian measures whose covariance kernels
corresponding to Euler and Wiener integrated processes.
 We briefly recall
Wiener and Euler integrated processes.

Let $W(t)$, $t\in[0,1]$, be a standard Wiener process, i.e. a
Gaussian random process with zero mean and covariance kernel
$$K^E_{1,0}(s,t)=K^W_{1,0}(s,t):= \min(s,t).$$ Consider two
sequences of integrated random processes $X^W_r$, $X^E_r$ on
$[0,1]$ defined inductively on $r$ by $X^W_0=X^E_0=W$ and for
$r=0,1,2,\dots$

$$X^E_{r+1}(t)=\int_{1-t}^1X^E_r(s)ds,$$
$$X^W_{r+1}(t)=\int_0^tX^W_r(s)ds.$$
The process  $\{X^E_r\}$ is called the univariate integrated Euler
process, while  $\{X^W_r\}$ is called the univariate integrated
Wiener process.

 Clearly, the corresponding Gaussian measures  to $X^W_r$ and  $X^E_r$ are concentrated on
a set of functions which are $r$ times continuously differentiable
but satisfy different boundary conditions.

The covariance kernel of $X^E_r$ is given by
$$K^E_{1,r}(x, y)=\int_{[0,1]^r}\min(x, s_1)\min(s_1, s_2)\dots\min(s_r, y)ds_1ds_2\dots ds_r$$
and is called the Euler kernel. The last kernel can be expressed
in terms of Euler polynomials. The covariance kernel of $X^W_r$ is
given by
$$K^W_{1,r}(x, y)=\int_0^{\min(x,y)}\frac{(x-u)^r}{r!}\frac{(y-u)^r}{r!}du$$
and is called the Wiener kernel.

The corresponding tensor product kernels on $[0,1]^d$ are given by
$$K^E_d(\x, \y)=\prod_{k=1}^dK^E_{1,r_k}(x_k, y_k)\ \  {\rm and}\ \ K^W_d(\x, \y)=\prod_{k=1}^dK^W_{1,r_k}(x_k, y_k).$$
Here $\{r_k\}_{k\in \Bbb N}$ is a sequence of nondecreasing
nonnegative integers
\begin{equation}\label{4.0}0\le r_1\le r_2\le r_3 \le \dots.\end{equation}They describe the increasing smoothness of a process with respect
to the successive coordinates.

For the problems $S$, the eigenvalues of the covariance operators
of the induced measures corresponding to Euler and Wiener
integrated processes are known (see \cite{GHT}):
$$\big\{\lz^{Y}_{d,j} \big\}_{j\in \Bbb N}=\big\{\lz^Y(1, j_1)\lz^Y(2, j_2)\dots\lz^Y(d, j_d) \big\}_{(j_1,\dots, j_d )\in\Bbb N^d},\ \ Y\in\{E,W\}.$$
where
\begin{equation}\label{4.1}\lz^E(k,j)=\bigg(\frac{1}{\pi(j-\frac{1}{2})}\bigg)^{2r_k+2},
\end{equation} for all $j\in \Bbb N$,  and
$$\lz^W(k,j)=\bigg(\frac{1}{\pi(j-\frac{1}{2})}\bigg)^{2r_k+2}+
\mathcal{O}(j^{-(2r_k+3)}),\quad\quad j\to\infty,$$ where for two
nonnegative sequences $f, g:\Bbb N\to [0,\infty)$,
$$f(k)=\mathcal{O}(g(k)),\ k\to\infty$$ means that there exists two
constants $C>0$ and $k_0\in \Bbb N$ for which $f(k)\le Cg(k)$
holds for any $k\ge k_0$, and $$f(k)=\Theta(g(k)),\ k\to\infty$$
means that $$f(k)=\mathcal{O}(g(k))\ \ {\rm and}\ \
g(k)=\mathcal{O}(f(k)),\ k\to\infty.$$

 Note that for all $k\in\Bbb N$,
$$f_k^E=h_k^E=\frac{\lz^E(k,2)}{\lz^E(k,1)}=\frac{1}{3^{2r_k+2}}.$$
In this case, we set $\tau_0\in (1/2,1)$. By \eqref{4.1} we have
\begin{align}&\quad\ \sup_{k\in \Bbb N}H^E(k,\tau_0)  =\sup_{k\in\Bbb
N}\sum_{j=2}^\infty\big(\frac{\lz^E(k,j)}{\lz^E(k,2)}\big)^{\tau_0}\notag \\
&=\sup_{k\in \Bbb N} \sum_{j=1}^\infty \Big(\frac
3{2j+1}\Big)^{\tau_0(2r_k+2)} \le \sum_{j=1}^\infty \Big(\frac
3{2j+1}\Big)^{2\tau_0}<\infty. \label{4.2}
\end{align}

It is proved in \cite{LPW2} that
\begin{equation*}
\begin{split}
&\lz^W(k,1)=\frac{1}{(r_k!)^2}\bigg(\frac{1}{(2r_k+2)(2r_k+1)}+\mathcal{O}(r_k^{-4})\bigg),\quad\quad k\to\infty,\\
&\lz^W(k,2)=\Theta\bigg(\frac{1}{(r_k!)^2r_k^4}\bigg),\quad\quad
k\to\infty.
\end{split}
\end{equation*}
Note that
\begin{equation*}h_k^W=\frac{\lz^W(k,2)}{\lz^W(k,1)}=\Theta(r_k^{-2})=\Theta((1+r_k)^{-2}),\ \quad
k\to\infty.\end{equation*} We conclude  that the problem $S$
corresponding to the Wiener integrated process satisfies Condition
(2) with $f_k^W=(1+r_k)^{-2}, k\in\Bbb N$.

From \cite[Thm. 4.1]{LPW2} it follows that for $\tau\in(3/5,1]$,
$$A_\tau:=\sup_{k\in\Bbb
N}\sum_{j=3}^\infty\big(\frac{\lz^W(k,j)}{\lz^W(k,2)}\big)^\tau<\infty.$$
This implies that for $\tau_0\in(3/5,1)$, $x\ge \tau_0$, we have
\begin{equation}\label{4.3}\sup_{k\in\Bbb N} H^W(k,\tau_0)= 1+\sup_{k\in\Bbb
N}\sum_{j=3}^\infty\big(\frac{\lz^W(k,j)}{\lz^W(k,2)}\big)^{\tau_0}
\le 1+A_{\tau_0}<\infty.
\end{equation}

 According to \eqref{4.2} and \eqref{4.3}, we know that the problems $S$
corresponding to the Euler and Wiener integrated processes satisfy
Conditions (2) and (3) with $f_k^E=3^{-2r_k-2}$  and $
f_k^W=(1+r_k)^{-2}$. By Remark 2.3, we have the following theorem.

\begin{thm} Consider the  problems $S=\{S_d\}$ in the average case setting with a zero mean
Gaussian measure whose covariance kernels corresponding to Euler
and Wiener integrated processes with the  smoothness $r_k$
satisfying \eqref{4.0}. Assume that $s>0$ and $t\in(0,1)$. Then
 for  NOR, we have

(1) for the Euler integrated process, $(s,t)$-WT holds iff
$$\lim\limits_{k\to\infty}k^{1-t}3^{-2r_k}(1+r_k)=0.$$

(2) for the Wiener integrated process, $(s,t)$-WT holds iff
$$\lim\limits_{k\to\infty}k^{1-t} (1+r_k)^{-2}\ln ^+ (1+r_k)=0.$$

\end{thm}

\vskip 2mm

 We recall tractability results of the above problems $S$ corresponding to Euler and Wiener integrated processes under the assumption
  \eqref{4.0} and using NOR. The sufficient and
necessary conditions for SPT, PT, QPT and WT  were obtained in
\cite{LPW2}, for UWT  in \cite{S2},  and for $(s,t)$-WT with $s>0$
and $t\ge 1$ in \cite{S3}. In \cite{S3}, Siedlecki also got the
sufficient conditions and  the necessary conditions  on $(s,t)$-WT
with $s>0$ and $t\in(0,1)$. However,  these conditions do not
completely match.
 Combining with our results, we have the
 following  results about the tractability of the above problem $S$
  using   NOR.

\vskip 2mm

\noindent{\bf  For the Euler integrated process under NOR:}

\vskip 3mm

 $\bullet$  SPT holds iff PT holds iff
$$\underset{k\to\infty}{\underline{\lim}}\frac{r_k}{\ln
k}>\frac{1}{2\ln3}.$$

\vskip 1mm

 $\bullet$  QTP holds iff $$\sup_{d\in \Bbb N}\frac{\sum_{k=1}^d(1+r_k)3^{-2r_k}}{\ln^+d}<\infty.$$

\vskip 1mm

 $\bullet$  UWT holds iff $$\underset{k\to\infty}{\underline{\lim}}\frac{r_k}{\ln k}\ge\frac{1}{2\ln3}.$$

\vskip 1mm

$\bullet$ $(s,t)$-WT with $s>0$ and $t>1$ always holds.

\vskip 3mm

 $\bullet$ $(s,1)$-WT with $s>0$ holds iff WT holds iff $$\lim\limits_{k\to \infty}r_k=\infty.$$

\vskip 1mm

$\bullet$  $(s,t)$-WT with $s>0$ and $t\in(0,1)$ holds iff
$$\lim\limits_{k\to\infty}k^{1-t}3^{-2r_k}(1+r_k)=0.$$

\vskip 1mm

\noindent{\bf For the Wiener integrated process under NOR:}

\vskip 3mm

$\bullet$  SPT holds iff PT holds iff
\begin{equation}\label{4.4}\underset{k\to\infty}{\underline{\lim}}\frac{\ln r_k}{\ln
k}>\frac{1}{2}.\end{equation}

\vskip 1mm

$\bullet$  QTP holds iff $$\sup_{d\in \Bbb
N}\frac{\sum_{k=1}^d(1+r_k)^{-2}\ln^+r_k}{\ln^+d}<\infty.$$

\vskip 1mm

$\bullet$ UWT holds iff
$$\underset{k\to\infty}{\underline{\lim}}\frac{\ln r_k}{\ln
k}\ge\frac{1}{2}.$$

\vskip 1mm

$\bullet$  $(s,t)$-WT with $s>0$ and $t>1$ always holds.

\vskip 3mm

$\bullet$ $(s,1)$-WT with $s>0$ holds iff WT holds iff
$$\lim_{k\to \infty}r_k=\infty.$$

\vskip 1mm

$\bullet$  $(s,t)$-WT with $s>0$ and $t\in(0,1)$ holds iff
$$\lim\limits_{k\to\infty}k^{1-t} (1+r_k)^{-2}\ln ^+(1+r_k)=0.$$

\vskip 3mm

\begin{rem} The authors in \cite{LPW2} obtained that the sufficient and
necessary condition  for SPT or PT under  NOR  is
$$\underset{k\to\infty}{\underline{\lim}}\frac{r_k}{k^v}>0 \text{
for some}\  v>\frac{1}{2}.$$ However, it is easy to verify that
this condition is equivalent to \eqref{4.4}.
\end{rem}

\subsection{Average-case $(s,t)$-WT with  Korobov kernels}

\

In this subsection we consider a multivariate approximation
problem $S=\{S_d\}$ defined over the space $C([0,1]^d)$ equipped
with a zero-mean Gaussian measure whose covariance kernel is given
as a Korobov kernel. Assume that the covariance kernel $K_d$ is of
product form,
$$K_d(\x,\y)=\prod_{k=1}^d\RR _k(x_k,y_k),\quad \x,\,\y\in [0,1]^d,$$
where $\RR _k=\RR _{r_k, g_k}$ are univariate Korobov kernels,
$$\RR _{\az,\beta}(x,y):=1+2\beta\sum_{j=1}^\infty j^{-2\az}\cos(2\pi j(x-y)),\quad x,\, y\in[0,1].$$
Here $\beta\in(0,1]$ is a scaling parameter, and $\az$ is a
smoothness parameter satisfying $\az>\frac{1}{2}$. Note that for
$x=y$ we have
$$\RR_{\az,\beta}(x,x)=1+2\beta \zeta(2\az),$$
where $\zeta(x)=\sum_{j=1}^\infty j^{-x}$ is the Riemann zeta
function which is well-defined only for $x>1$. We assume that
$\big\{r_k\big\}_{k\in\Bbb N}$ and $\big\{g_k\big\}_{k\in\Bbb N}$
satisfy
\begin{equation}\label{4.20}1\ge g_1\ge g_2 \ge \dots\ge g_k\ge \dots >0,\end{equation}and \begin{equation}\label{4.21} r_*:=\inf_{k\in \Bbb N}
r_k>\frac{1}{2}.\end{equation}

For the problem $S=\{S_d\}$, the eigenvalues of the covariance
operator $C_{\nu_d}$ of the induced measure  are known, see
\cite{LPW1}.
$$\big\{\lz_{d,j} \big\}_{j\in \Bbb N}=\big\{\lz(1, j_1)\lz(2, j_2)\dots\lz(d, j_d) \big\}_{(j_1,\dots, j_d )\in\Bbb N^d},$$
where $\lz(k,1)=1$ and
$$\lz(k,2j)=\lz(k,2j+1)=\frac{g_k}{j^{2r_k}},\quad j\in\Bbb N.$$

In this case, we set $\tau_0\in (\frac1{2r*},1)$. We have
\begin{align*}\sup_{k\in\Bbb N}H(k,\tau_0)&=\sup_{k\in\Bbb N}
\sum_{j=2}^\infty\big(\frac{\lz(k,j)}{\lz(k,2)}\big)^{\tau_0}=2\sup_{k\in\Bbb
N}\sum_{j=1}^\infty j^{-2r_kx}\\ &\le 2\sum_{j=1}^\infty
j^{-2r_*\tau_0}=2\zeta(2r_*\tau_0)<\infty. \end{align*} This means
that the problem $S$ has Property (P) with $f_k=g_k$. By Theorem
2.1, we have the following theorem.

\begin{thm} Consider the  problem $S=\{S_d\}$ in the average case setting with a zero mean
Gaussian measure whose covariance operator is given as the Korobov
kernel with the scale $ g_k$ and smoothness $r_k$ satisfying
\eqref{4.20} and \eqref{4.21}, respectively. Assume that $s>0$ and
$t\in(0,1)$. Then $S$ is $(s,t)$-WT for ABS or NOR iff
$$\lim\limits_{k\to\infty}k^{1-t}g_k\ln^+\frac1{g_k}=0. $$\end{thm}

\begin{rem}Using the method of \cite{S3, LX2}, we can get easily
that for the above problem $S$ under ABS or NOR, $(s,t)$-WT always
  holds with $s>0$ and $t>1$, and $(s,1)$-WT with
$s>0$ holds iff
  WT holds iff $
\lim\limits_{k\to\infty}g_k=0.$ \end{rem}

  We recall tractability results of the above problem $S$. In \cite{LPW1, X1, X2}, the authors considered the problem $S$ under the assumption \eqref{4.20} and
\begin{equation}\label{4.22}1/2<r_1\le r_2\le\dots\le r_k\le \dots.\end{equation}
   However, there is no need to assume monotonicity for the smoothness parameters $r_k,\ k\in\Bbb N$. Indeed, it suffices to assume \eqref{4.21} instead of \eqref{4.22}.
   The sufficient and
necessary conditions for  SPT, PT, WT
  under NOR were given in \cite{LPW1},
 for QPT under NOR  in \cite{LPW1, X1, K}, and  for UWT under ABS or NOR  in
 \cite{X2}. Combining with our results, we have the
 following  results about the tractability of the problem $S$
  using  ABS and NOR.

\vskip 2mm

$\bullet$  For NOR or ABS, SPT holds iff PT holds iff
\begin{equation}\label{4.23}\underset{j\to\infty}{\underline{\lim}}\frac{\ln \frac
1{g_j}}{\ln j}> 1.\end{equation}

\vskip 1mm

$\bullet$  For NOR, QPT holds iff $$\sup_{d\in \Bbb N}\frac{1}{\ln
^+d}\sum_{k=1}^dg_k\ln^+\frac{1}{g_k}<\infty.$$

\vskip 1mm

$\bullet$ For ABS or NOR,  UWT holds  iff
\begin{equation}\label{4.24}\underset{j\to\infty}{\underline{\lim}}\frac{\ln \frac
1{g_j}}{\ln j}\ge  1.\end{equation}

\vskip 1mm

$\bullet$ For ABS or NOR, $(s,t)$-WT with $s>0$ and $t>1$ always
holds.

\vskip 3mm

$\bullet$  For ABS or NOR, $(s,1)$-WT with $s>0$ holds iff WT
holds iff $\lim\limits_{j\to \infty}g_j=0$.

\vskip 3mm

$\bullet$  For ABS or NOR, $(s,t)$-WT with $s>0$ and $t\in(0,1)$
holds iff $$
\lim\limits_{k\to\infty}k^{1-t}g_k\ln^+\frac1{g_k}=0.$$ \vskip 1mm

\begin{rem} In
 \cite{X2}, Xu obtained that the sufficient and
necessary condition  for UWT under ABS or NOR  is
$\lim\limits_{j\to\infty}g_jj^p=0$ for all $p\in(0,1)$. This
condition is equivalent to \eqref{4.24}.
\end{rem}

\begin{rem} In \cite{LPW1}, the sufficient and
necessary condition for  SPT or PT
 only under NOR was given. However, this condition is also true for ABS. Indeed, due to \eqref{2.1-0}, it suffices to prove that SPT holds for ABS  if \eqref{4.23} holds.
We assume that \eqref{4.23} holds.
  Then $\sum\limits_{k=1}^\infty g_k<\infty$. This means that
$$e(0,d)=\exp\Big(\frac 12\ln \Big(\sum_{k=1}^\infty
\lz_{d,k}\Big)\Big)\le \exp\Big(\frac{A_1}2\sum_{k=1}^dg_k\Big)\le
\exp\Big(\frac{A_1}2\sum_{k=1}^\infty g_k\Big)=:B<\infty,
$$where   in the first inequality  we
used \eqref{2.1-1}, $ A_1=2\zeta(2r_*\tau_0).$
 From \cite{LPW1} we know that SPT holds for
NOR. Using the inequalities
$$ n^{\rm ABS}(\vz, S_d)=  n^{\rm NOR}(\frac{\vz}{e(0,d)}, S_d)\le  n^{\rm NOR}(\frac{\vz}B, S_d),
$$we get that  SPT holds for NOR iff SPT holds for ABS. Hence,
 SPT  for ABS holds.
 \end{rem}

\subsection{Average-case $(s,t)$-WT with analytic Korobov kernels}

\

In this subsection we consider a multivariate approximation
problem $S=\{S_d\}$ defined over the space of $C([0,1]^d)$
equipped with a zero-mean Gaussian measure whose covariance kernel
is given as an analytic Korobov kernel. Assume that the covariance
kernel $K_d$ is of product form,
$$K_d(\x,\y)=\prod_{k=1}^d K _{1,a_k,b_k}(x_k,y_k),\quad \x,\,\y\in
[0,1]^d,$$ where $ K _{1,a_k,b_k}$ are univariate analytic Korobov
kernels,
$$K_{1,a,b}(x,y)=\sum_{\h\in\Bbb Z}\oz^{a|h|^b}\exp(2\pi i h(x-y)), \
\ x, y\in[0,1].$$Here $\oz\in(0,1)$ is a fixed number,
$i=\sqrt{-1}$,  $a,\, b>0$. Hence, we have
\begin{equation*}K_d(\x, \y)=\sum_{\h\in\Bbb Z^d}\oz_\h\exp(2\pi i\h(\x-\y)),\ \ \x, \y\in[0,1]^d,\end{equation*}
with $$\oz_\h=\oz^{\sum_{k=1}^da_k|h_k|^{b_k}},\ \ \forall\
\h=(h_1,h_2,\dots,h_d)\in\Bbb Z^d,$$ for fixed $\oz\in(0,1)$.

 We assume that the sequences
${\bf a}=\big\{a_k\big\}_{k\in\Bbb N}$ and ${\bf
b}=\big\{b_k\big\}_{k\in\Bbb N}$ satisfy
\begin{equation}\label{4.25}0<a_1\le a_2 \le \dots\le a_k\le \dots ,\ \ {\rm and}\ \ b_*:=\inf_{k\in \Bbb N}
b_k>0.\end{equation}

For the above problem $S=\{S_d\}$, the eigenvalues of the
covariance operator $C_{\nu_d}$ of the induced measure $\nu_d$
are given by
$$\big\{\lz_{d,j} \big\}_{j\in \Bbb N}=\big\{\lz(1, j_1)\lz(2, j_2)\dots\lz(d, j_d) \big\}_{(j_1,\dots, j_d )\in\Bbb N^d},$$
where $\lz(k,1)=1$, and
$$\lz(k,2j)=\lz(k,2j+1)=\oz^{a_kj^{b_k}},\quad j\in\Bbb N.$$

In this case, we set $\tau_0\in (0,1)$. We have
$$ \sup_{k\in\Bbb N}H(k,\tau_0)=\sup_{k\in\Bbb N}\sum_{j=2}^\infty\big(\frac{\lz(k,j)}{\lz(k,2)}\big)^{\tau_0}=
2\sup_{k\in\Bbb N}\sum_{j=1}^\infty \oz^{\tau_0 a_k(j^{b_k}-1)}\le
2\sum_{j=1}^\infty \oz^{\tau_0a_1(j^{b_*}-1)}. $$ Since
$$\oz^{\tau_0a_1(j^{b_*}-1)}=j^{-\frac{\tau_0a_1(j^{b_*}-1)\ln\frac1\oz}{\ln
j}},\ \ {\rm and}\ \
\lim_{j\to\infty}\frac{\tau_0a_1(j^{b_*}-1)\ln\frac1\oz}{\ln
j}=\infty,
$$We get that
$$M_{\tau_0}:=2\sum_{j=1}^\infty \oz^{\tau_0a_1(j^{b_*}-1)}<\infty.$$This means that the problem $S$ has Property (P) with
$f_k=\oz^{a_k}$. By Theorem 2.1, we have the following theorem.

\begin{thm} Consider the  problem $S=\{S_d\}$ in the average case setting with a zero mean
Gaussian measure whose covariance operator is given as the
analytic  Korobov kernel with the sequences ${\bf a}$  and ${\bf
b}$ satisfying \eqref{4.25}. Assume that $s>0$ and $t\in(0,1)$.
Then $S$ is $(s,t)$-WT for ABS or NOR iff
$$\lim\limits_{k\to\infty}k^{1-t}a_k\oz^{a_k}=0. $$\end{thm}

  We recall tractability results of the above problem $S$ under the assumption
  \eqref{4.25}.
   The sufficient and
necessary conditions for  SPT, PT, UWT, WT
  under NOR or ABS,  and for QPT under NOR were given in
  \cite{LX}, and for $(s,t)$-WT with $s>0$ and $t\ge 1$
 under ABS or NOR  in
 \cite{X2}. However, the authors did not find out the
matching necessary and sufficient conditions on $(s,t)$-WT with
$s>0$ and $t\in(0,1)$ under ABS or NOR.
 Combining with our results, we have the
 following  results about the tractability of the above problem $S$
  using  ABS and NOR.

\vskip 2mm

$\bullet$  For NOR or ABS, SPT holds iff PT holds iff
\begin{equation*}\underset{j\to\infty}{\underline{\lim}}\frac{a_j}{\ln j}> \frac{1}{\ln \oz^{-1}}.\end{equation*}

\vskip 1mm

$\bullet$  For NOR, QPT holds iff $$\sup_{d\in \Bbb N}\frac{1}{\ln
^+d}\sum_{k=1}^da_k\oz ^{a_k} <\infty.$$

\vskip 1mm

$\bullet$ For ABS or NOR,  UWT holds  iff
\begin{equation}\label{4.26}\underset{j\to\infty}{\underline{\lim}}\frac{a_j}{\ln j}\ge \frac{1}{\ln \oz^{-1}}.\end{equation}

\vskip 1mm

$\bullet$ For ABS or NOR, $(s,t)$-WT with $s>0$ and $t>1$ always
holds.

\vskip 3mm

$\bullet$  For ABS or NOR, $(s,1)$-WT with $s>0$ holds iff WT
holds iff $\lim\limits_{j\to \infty}a_j=\infty$.

\vskip 3mm

$\bullet$  For ABS or NOR, $(s,t)$-WT with $s>0$ and $t\in(0,1)$
holds iff $$ \lim\limits_{k\to\infty}k^{1-t}a_k\oz^{a_k}=0.$$
\vskip 1mm

\begin{rem} In
 \cite{X2}, Xu obtained that the sufficient and
necessary condition  for UWT under ABS or NOR  is
$\lim\limits_{j\to\infty}\oz^{a_j}j^p=0$ for all $p\in(0,1)$. This
condition is equivalent to \eqref{4.26}.
\end{rem}

\end{document}